\documentclass[12pt]{article}
\usepackage{amssymb, amsmath, latexsym, txfonts, mathabx, a4}
\usepackage[latin1]{inputenc}
\usepackage{graphicx}
\usepackage[all]{xy}
\usepackage[usenames]{color}
\usepackage{tikz}
\newenvironment{proof}[1][Proof]{\noindent\textbf{#1.} }{\ \rule{0.5em}{0.5em}}

\newtheorem{De}{Definition}[section]

\newtheorem{Pro}[De]{Proposition}
\newtheorem{Le}[De]{Lemma}
\newtheorem{Co}[De]{Corollary}
\newtheorem{Rem}[De]{Remark}
\newtheorem{Ex}[De]{Example}

\renewcommand{\ker}{\ensuremath{\mathsf{Ker\,}}}

\newcommand{\Lieg}{\ensuremath{\mathfrak{g}}}

\newcommand{\LieU}{\ensuremath{\mathfrak{U}}}

\newbox\pullbackbox
\setbox\pullbackbox=\hbox{\xy 0;<1mm,0mm>: \POS(4,0)\ar@{-} (0,0) \ar@{-} (4,4)
\endxy}

\newdir{>>}{{}*!/3.5pt/:(1,-.2)@^{>}*!/3.5pt/:(1,+.2)@_{>}*!/7pt/:(1,-.2)@^{>}*!/7pt/:(1,+.2)@_{>}}
\newdir{ >>}{{}*!/8pt/@{|}*!/3.5pt/:(1,-.2)@^{>}*!/3.5pt/:(1,+.2)@_{>}}
\newdir{ |>}{{}*!/-3.5pt/@{|}*!/-8pt/:(1,-.2)@^{>}*!/-8pt/:(1,+.2)@_{>}}
\newdir{ >}{{}*!/-8pt/@{>}}
\newdir{>}{{}*:(1,-.2)@^{>}*:(1,+.2)@_{>}}
\newdir{<}{{}*:(1,+.2)@^{<}*:(1,-.2)@_{<}}

\newcommand{\g}{\frak g}

\begin{document}

\centerline{\bf  From Trigroups to Leibniz 3-algebras}

\bigskip
\bigskip
\centerline{\bf Guy R. Biyogmam }

\bigskip
\centerline{Department of Mathematics, Georgia College \& State University}
\centerline{Campus Box 17 Milledgeville, GA 31061-0490}
\centerline{ {E-mail address}: guy.biyogmam@gcsu.edu}
\bigskip

\centerline{\bf Calvin Tcheka }
\bigskip
\centerline{Department of Mathematics, Faculty of sciences-University of Dschang}
\centerline{Campus Box (237)67 Dschang, Cameroon}
\centerline{ {E-mail address}:  jtcheka@yahoo.fr}
\bigskip

\bigskip

\date{}

\bigskip \bigskip \bigskip

{\bf Abstract:}
In this paper, we study the category of trigroups as a generalization of the notion of digroup \cite{MK} and analyze their relationship with 3-racks \cite{BGR1} and Leibniz 3-algebras \cite{CLP}. Trigroups are essentially  associative trioids in which there are bar-units and bar-inverses. We prove that    3-racks can be constructed by conjugating  trigroups. We also prove that trigroups equipped with a smooth manifold structure produce Leibniz 3-algebras via their associated Lie 3-racks.

{\bf 2010 MSC:} 17A99; 20M99.
\bigskip

{\bf Key words:}  Digroup, 3-rack, Trigroup, Leibniz 3-algebra.
\bigskip

\section{Introduction}
An axiomatic definition of the concept of digroups was introduced by M. Kinyon in \cite{MK} as a generalization of groups  in his partial solution to the coquecigrue problem, which consists of generalizing Lie's third theorem to Leibniz algebras \cite{Lo2}. Other axiomatic descriptions of digroups were independently studied by R. Felipe in \cite{FR2} and R. Liu in \cite{KL1} . It is worth mentioning that prior to these axiomatic definitions, the notion of digroups already appeared implicitly in Loday's work on dialgebras \cite{Lo3}. Similarly to how digroups are related to groups, a trigroup is a set A endowed with 3 binary operations $\vdash,\perp$ and $\dashv$ so that $(A,\vdash,\dashv)$ is a digroup, and $(A,\vdash,\perp)$  and $(A,\perp,\dashv)$ are disemigroups  in which the operations are compatible with bar-units and appropriate inverses. In this paper, we generalize the conjugation of digroups to trigroups and show that every trigroup $A$ is equipped with a pointed 3-rack structure, and thus produces a pointed rack structure on $A\times A$  by \cite[Example 2.6]{BGR1}. When trigroups are also smooth manifolds, their associated 3-racks inherit the smooth manifold structure, and  produce Leibniz 3- algebras thanks to \cite[Corollary 3.6]{BGR1}.  Another  analysis of a relationship between Leibniz 3-algebras and an algebra with 3 associative operations, namely  trialgebras \cite[Definition 1.1]{Lo3}, was conducted in \cite{cas1}. Similarly to digroups, trigroups along with their homomorphisms constitute a pointed category, which generalizes both the category of groups and the category of digroups.  This category is a subcategory of the category of associative trioid introduced by J. L. Loday and M. O. Ronco  in \cite[Definition 1.1]{Lo3}.

This paper is organized as follows: In section 2, we present the axioms of the known notion of disemigroup in a way suitable  to define the notion of trisemigroup, then  we provide a construction of trisemigroups from given disemigroups. In section 3, we study the notion of trigroups and prove several properties related to their operations. In section 4, we use these properties to  construct a functor from the category of trigroups to the category of Leibniz 3-algebras.

\section{Preliminaries } \label{preliminaries}
\subsection{Disemigroups}

\begin{De}
A  left disemigroup $(A,\vdash,\star)$ is a set $A$ together with two binary operations $\vdash$ and $\star$ such that $(A,\vdash)$ and $(A,\star)$  are semigroups satisfying the following relations:
\begin{enumerate}

\item[(L1)]  $x\vdash( y\vdash z)=(x\star y)\vdash z$
\item[(L2)] $x\vdash (y\star z)=(x\vdash y)\star z$
\end{enumerate}
for all $x,y,z\in A.$
\end{De}

\begin{Rem}\label{re1}
$(x\star (y\vdash z))\vdash t=(x\star y)\vdash (z\vdash t)$ for all $x,y,z,t\in A.$
\end{Rem}
\begin{proof}
We use (L1) and the fact that the operation $\vdash$ is associative:

$\begin{aligned}
(x\star (y\vdash z))\vdash t &=x\vdash( (y\vdash z)\vdash t) =(x\vdash (y\vdash z))\vdash t \\&=((x\star y)\vdash z)\vdash t =(x\star y)\vdash (z\vdash t).\end{aligned}$
\end{proof}
\begin{De}
A  right disemigroup $(A,\star,\dashv)$ is a set $A$ together with two binary operations  $\star$ and $\dashv$ such that $(A,\star)$ and $(A,\dashv)$  are semigroups satisfying the following relations:
\begin{enumerate}
\item[(R1)]  $x\dashv (y\dashv z)=x\dashv(y\star z)$
\item[(R2)] $(x\star y)\dashv z=x\star (y\dashv z)$

\end{enumerate}
for all $x,y,z\in A.$
\end{De}
Similarly to Remark \ref{re1}, we have 
\begin{Rem}
$x\dashv((y\dashv z)\star t)=(x\dashv y)\dashv (z\star t)$ for all $x,y,z,t\in A.$
\end{Rem}

\begin{Rem}
 $(A,\vdash, \dashv)$ is a disemigroup (\cite[Definition 4.1]{MK}) if it is both a left disemigroup and a right disemigroup.

\end{Rem}

\begin{De}
A disemigroup $(A,\vdash,\dashv)$ is a   left digroup (resp. right digroup)   if there exists an element $1\in A$ such that 
 $1\vdash x=x$ (resp. $x\dashv 1=x$) for all $x\in A,$ and 
 for all $x\in A,$ there exists $x^{-1}\in A$ satisfying  $x\vdash x^{-1}=1$ (resp. $ x^{-1}\dashv x=1$).
\end{De}

\begin{Rem}
A set $(A,\vdash,\dashv)$ is a  digroup  if and only if it is a left digroup and a right digroup. 
\end{Rem}
\subsection{Trisemigroups}
The following generalizes the definition of disemigroups to a ternary algebra.
\begin{De}
A  trisemigroup $(A,\vdash,\perp,\dashv)$ is a set $A$ equipped with three binary operations  $\vdash,$ $\perp$  and $\dashv$ respectively called left, middle and right, and satisfying the following conditions:
\begin{enumerate}
\item[(T1)]  $(A,\vdash, \dashv)$  is a disemigroup
\item[(T2)] $(A,\vdash,\perp)$ is a left disemigroup
\item[(T3)]  $(A,\perp,\dashv)$ is a right disemigroup
\item[(T4)] $(x\dashv y)\perp z=x\perp (y\vdash z)$ for all $x,y,z\in A.$
\end{enumerate}
\end{De}

Note that there are 11 axioms in the conditions $(T1), (T2), (T3)$ and  $(T4).$ These axioms are exactly the 11 relations of the definition of an associative trioid introduced by Loday and Ronco \cite[Defintion 1.1]{Lo3}. It is worth mentioning that  J. D. Phillips's  work  \cite{JDP} on digroups reduces these 11 axioms to 7 axioms. In this paper, we use the terminology "trisemigroup" instead of "associative trioid" to remain in the semigroup jargon. Also, we use the notation $"ass"$ to refer to the associativity of the operations $\vdash,\perp,\dashv.$
\subsection{From Disemigroups to Trisemigroups}
Using J. D. Phillips's  work  \cite{JDP} on digroups, it is clear that  a trisemigroup  $(A,\vdash,\perp,\dashv)$ is equipped with at least 3 disemigroup structures, namely: $(A,\vdash,\dashv),~$ $(A,\vdash,\perp)$ and $(A,\perp,\dashv). $ 
Now, let $G$ be a set endowed with two binary operations $\vdash, \dashv$. Then define on $G\times G$ the following binary operations: 
\begin{enumerate}
\item[a)] $(u,h) \Vdash (v,k) = (h\vdash v, h\vdash k)$  
\item[b)] $(u,h)\dashV (v,k) = (u,h\dashv k)$ 
\item[c)] $(u,h) \Perp (v,k) =(h\dashv v, h\dashv k)$
 \end{enumerate}
for all $u,v , h,k \in G.$ 

Then we have the following:
\begin{enumerate}
\item $(u,h)\dashV((v,k)\dashV (w,l))=(u,h)\dashV (v,k\dashv l)=(u,h\dashv (k\dashv l))$
\item $(u,h)\dashV((v,k)\Perp(w,l))=(u,h)\dashV (k\dashv w, k\dashv l)=(u,h\dashv (k\dashv l))$

\item $((u,h)\Perp(v,k))\dashV (w,l)=(h\dashv v,h\dashv k)\dashV (w,l)=(h\dashv v, h\dashv (k\dashv  l))$

\item $(u,h)\Perp((v,k)\dashV (w,l))=(u,h)\Perp (v,k\dashv l)=(h\dashv v, h\dashv (k\dashv  l))$

\item $(u,h)\Vdash((v,k)\Vdash(w,l))=(u,h)\Vdash (k\vdash w,k\vdash l)=(h\vdash (k\vdash w),h\vdash (k\vdash l))$

\item $((u,h)\Perp(v,k))\Vdash(w,l)=(h\dashv v,h\dashv k)\Vdash (w,l)=((h\dashv k)\vdash w),(h\dashv k)\vdash l))$

\item $(u,h)\Vdash((v,k)\Perp(w,l))=(u,h)\Vdash (k\dashv w,k\dashv l)=(h\vdash (k\dashv w),h\vdash (k\dashv l))$

\item $((u,h)\Vdash(v,k))\Perp(w,l)=(h\vdash v,h\vdash k)\Perp (w,l)=((h\vdash k)\dashv w),(h\vdash k)\dashv l))$

\item $((u,h)\dashV(v,k))\Perp(w,l)=(u, h\dashv k)\Perp (w,l)=((h\dashv k)\dashv w),(h\dashv k)\dashv l))$

\item $(u,h)\Perp((v,k)\Vdash(w,l))=(u,h)\Perp (k\vdash w,k\vdash l)=(h\dashv (k\vdash w),h\dashv (k\vdash l))$

\item $((u,h)\dashV(v,k))\Vdash(w,l)=(u,h\dashv k)\Vdash (w,l)=((h\dashv k)\vdash w),(h\dashv k)\vdash l))$

\item $(u,h)\Vdash((v,k)\dashV(w,l))=(u,h)\Vdash(v,k\dashv l)=(h\vdash v,h\vdash( k\dashv l))$
\item $((u,h)\Vdash(v,k))\dashV(w,l)=(h\vdash v,h\vdash k)\dashV (w,l)=(h\vdash v,(h\vdash k)\dashv l))$
\item $(u,h)\dashV((v,k)\Vdash(w,l))=(u, h)\dashV(k\vdash w,k\vdash l)=(u,h\dashv (k\vdash l))$

\end{enumerate}

As a consequence, we have the following results:

\begin{Pro}
Let $G$ be a set endowed with two binary operations $\vdash, \dashv.$

\begin{enumerate}

\item[a)] $(G\times G, \Perp,\dashV)$ is a right disemigroup.
\item[b)] If  $( G, \vdash,\dashv)$ is a disemigroup, then $(G\times G, \Vdash,\dashV)$ is a disemigroup.
\item[c)] If  $( G, \vdash,\dashv)$ is a left  disemigroup, then $(G\times G, \Vdash,\Perp)$ is a left disemigroup.
\item[d)] If  $( G, \vdash,\dashv)$ is a left  disemigroup, then $(G\times G, \Vdash,\Perp,\dashV)$ is a trisemigroup.

\end{enumerate}
\end{Pro}

\begin{proof} It is easy to verify that $\Vdash,$  $\dashV$ and $\Perp$ are associative whenever  $\vdash,$  $\dashv$ and $\perp$ are associative.
For a),  the axioms (R1) and (R2) are always satisfied for $\dashV$ and $\Perp$ by $1.,2.,3.$ and $4.$ For b),  the axiom (L1) is  satisfied for $\Vdash$ and $\dashV$ by $5.$ and $11.$ whenever they are satisfied for $\vdash$ and $\dashv.$ The axiom (R1) is  satisfied for $\Vdash$ and $\dashV$ by $1.$ and $14.$ whenever they are satisfied for $\vdash$ and $\dashv.$ Also  the axioms (L2) and (R2) are  satisfied for $\Vdash$ and $\dashV$ by $12.$ and $13.$ whenever they are satisfied for $\vdash$ and $\dashv.$ For c), the axioms (L1) and (L2) are  satisfied for $\Vdash$ and $\Perp$ by $5.,6.,7.$ and $8.$ whenever they are satisfied for $\vdash$ and $\dashv.$  For d), It remains to verify (T4).  It is verified by $9.$ and $10.$ whenever (R1) holds for $\vdash$ and $\dashv.$
\end{proof}

\section{Trigroups}
In this section, we introduce the notion of trigroups and study several properties on the conjugation operation on them. 

\begin{De}
A  trisemigroup $A$ is a trimonoid if there exists an element $1\in A$ such that 
  $$1\vdash x= x=x\dashv 1~~\mbox{for all }~x\in A.~~~~~~~~~~~~~~~~~~(I)$$
Note that the distinguish element $1\in A$ satisfying $(I)$ may not be unique. Set $$\LieU_A:=\{e\in A: e\vdash x=x=x\dashv e~~\mbox{for all}~x,y\in A\}.$$ This set is referred to as the set of bar-units in $A.$

A  trimonoid  is a trigroup if  
 for all $x\in A,$ there exists $x^{-1}\in A$ (called inverse of $x$) such that $$x\vdash x^{-1}=1=x^{-1}\dashv x~~\mbox{and} ~~x\perp x^{-1}=1=x^{-1}\perp x.$$

A morphism between two trigroups is a map that preserves the 3 binary operations and is compatible with bar-units and inverses.
\end{De}

\begin{Rem}
\noindent
\begin{itemize}
\item[a)] If $(A, \vdash,\perp,\dashv)$ is a trigroup, then $(A, \vdash,\dashv)$ is a digroup. In other words, there is a forgetful functor $Trig\to Dig$ from the category of trigroups to the category of digroups. 
\item[b)] $e$ is a bar-unit in $(A, \vdash,\perp,\dashv)$ if and only if $e$ is a bar-unit in the underlying digroup $(A, \vdash,\dashv).$

\item[c)] If $\vdash=\perp=\dashv$, then we simply have a group. So we get a functor from the category of groups to the category of trigroups. Moreover, we may regard the trivial group as a trigroup, and thus
 a zero object in the category of trigroups.
\end{itemize}
\end{Rem}
\begin{Ex} \label{ExTri}
Let $M$ be a set  and  $H$  a group acting  on the left of $M.$ Suppose that there exists  $e\in M$ satisfying  $he = e$ for all $h \in H,$ and  $H$ acts transitively on $M-\{e\}.$ Now define on  $A := M \times H$ , 
 the following binary operations:
\begin{enumerate}
\item[i)]  $(u,h){\vdash} (v,k)=(hv,hk)$
\item[ii)]  $(u,h){\dashv} (v,k)=(u,hk)$
\item[iii)]  $(u,h){\perp} (v,k)=(e,hk)$
\end{enumerate}
for all $u,v\in M$ and $h,k\in H.$ One readily shows that $(A,\vdash,\perp,\dashv)$ is a trigroup with distinguish bar-unit  $(e,1)$ in which $(e,h^{-1})$ is the inverse of $(u,h).$ 
\end{Ex}

\begin{Ex}
Let $Z(\mathbb{K},n)$ be the center of $GL(n,\mathbb{K}),$  the general linear group of degree $n$ with coefficients in  $\mathbb{K}.$ 
Define on $A:=\mathbb{K}^n\times Z(n,\mathbb{K})$ the following binary operations:
\begin{enumerate}
\item[i)]  $(x,M){\vdash} (y,N)=(My,MN)$
\item[ii)]  $(x,M){\dashv} (y,N)=(x,MN)$
\item[iii)]  $(x,M){\perp} (y,N)=(0,MN)$
\end{enumerate}
for all $x,y\in \mathbb{K}^n$ and $M,N\in Z(n,\mathbb{K}).$ Then by Example \ref{ExTri}, $(A,\vdash,\perp,\dashv)$ is a trigroup with distinguish bar-unit  $(0,I_n)$ in which $(0, M^{-1})$ is the inverse of $(x,M),$ where $I_n$ is the identity matrix.
\end{Ex}


\begin{Le}\label{inverse}
Let $(A,\vdash,\perp,\dashv )$  be a trigroup. 
Then the following is true:
\begin{enumerate}

\item[1)]  $x\vdash 1=1\perp x=x\perp 1=1\dashv x= (x^{-1})^{-1}$ for all $x\in A.$
\item[3)] $(x\perp y)^{-1}=y^{-1}\perp x^{-1}$ for all $x,y\in A.$ 
\item[4)]  $(x\vdash y)^{-1}=y^{-1}\vdash x^{-1}=y^{-1}\dashv x^{-1}=(x\dashv y)^{-1}$ for all $x,y\in A.$   Consequently, $((x^{-1})^{-1})^{-1}=x^{-1}.$ 
\item[2)] $x^{-1}\vdash x\vdash y=x\vdash x^{-1}\vdash y=y~~$ for all $x,y\in A.$ 

\item[5)]  The set $J=\{x^{-1}: x\in A\}$ is a group in which $\vdash=\perp=\dashv.$ This produces a functor from the category of trigroups to the category of groups.

\item[6)]  The mapping  $\phi: A\to J$ defined by $x\mapsto (x^{-1})^{-1}$ is an epimorphism of trigroups that fixes $J$, and $\ker\phi=\LieU_A.$ 

\end{enumerate}
\end{Le}

\begin{proof}
 Let $x\in A.$ Then 
 $$x\vdash 1=x\vdash (x^{-1}\perp x) \stackrel{L2}=(x\vdash x^{-1})\perp x=1\perp x .$$ Similarly, one proves  $1\dashv x=x\perp 1.$ In particular, $1\perp 1=1\vdash 1=1\dashv 1=1.$ Also,
 
 $
 \begin{aligned}
 (x\perp 1)\perp x^{-1}&=(x\perp (x^{-1}\dashv x)) \perp x^{-1}\stackrel{R2}=((x\perp x^{-1})\dashv x) \perp x^{-1}\\&=(1\dashv x )\perp x^{-1}\stackrel{T4}=1\perp (x\vdash x^{-1})= 1\perp 1=1.
 \end{aligned} $
 
 Similarly, one proves  that $x^{-1}\perp (1\perp x)=1.$  
Therefore, 1) follows by \cite[Lemma 4.3(2)]{MK} and  \cite[Lemma 4.5(1)]{MK} . 
2) follows from \cite[Lemma 4.5]{MK} since $(A,\vdash,\dashv)$ is a digroup.  To prove 3), let $x,y\in A.$ Then 
 $$(y^{-1}\perp x^{-1})\perp(x\perp y)\stackrel{ass}=(y^{-1}\perp (x^{-1}\perp x))\perp y \stackrel{ass}=y^{-1}\perp (1\perp y) \stackrel{1)}=1.
$$
    Similarly, one proves that $(x\perp y)\perp(y^{-1}\perp x^{-1})=1$ . 
Also 

 $
 \begin{aligned}
(x\perp y)\vdash(y^{-1}\perp x^{-1})&\stackrel{L1}=x\vdash (y\vdash(y^{-1}\perp x^{-1}))\stackrel{L2}=x\vdash ((y\vdash y^{-1})\perp x^{-1}))\\&=x\vdash(1\perp x^{-1})\stackrel{L2}=(x\vdash 1)\perp x^{-1}\stackrel{ 1)}=(1\dashv x )\perp x^{-1}\\&\stackrel{T4}=1\perp (x\vdash x^{-1})= 1\perp 1=1. \end{aligned} $

Similarly, one shows that $(y^{-1}\perp x^{-1})\dashv(x\perp y)=1.$\\
For 4), note that by \cite[Lemma 4.5(2)]{MK}, It is enough to show that  \\$(y^{-1}\vdash x^{-1})\perp (x\vdash y)=1$ and $(x\dashv y)\perp(y^{-1}\dashv x^{-1}) =1.$ Indeed, 

 $
 \begin{aligned}
 (y^{-1}\vdash x^{-1})\perp (x\vdash y)&\stackrel{L2}=y^{-1}\vdash (x^{-1}\perp (x\vdash y))\stackrel{T4}=y^{-1}\vdash ((x^{-1}\dashv x)\perp y))\\&=y^{-1}\vdash (1\perp y))\stackrel{1)}=y^{-1}\vdash (y\vdash 1))\stackrel{L1}=(y^{-1}\perp y)\vdash1\\&=1\vdash 1=1.\end{aligned} $
 
  The proof that  $(x\dashv y)\perp(y^{-1}\dashv x^{-1}) =1$ is similar.  The consequence above mentioned follows as $((x^{-1})^{-1})^{-1}=(x\vdash 1)^{-1}=1\vdash x^{-1}=x^{-1}$ due to 1). 
\\
For 5), note that  since $(A,\vdash,\dashv)$ is a digroup, it follows by  \cite[Lemm 4.5]{MK} that $J$ is a group in which $\vdash=\dashv.$ It remains to show that $\perp=\vdash.$ Indeed, for all $x,y\in A,$ 
 
 $
 \begin{aligned}
 (x^{-1}\perp y^{-1})\dashv (y\vdash x)&\stackrel{T4}= (x^{-1}\perp y^{-1})\dashv (y\dashv x) \stackrel{ass}= ((x^{-1}\perp y^{-1})\dashv y)\dashv x\\& \stackrel{R2}=(x^{-1}\perp (y^{-1}\dashv y))\dashv x=(x^{-1}\perp 1)\dashv x\\& \stackrel{R2}=x^{-1}\perp (1\dashv x)\\&=1 ~~\mbox{ since} ~~(x^{-1})^{-1}=x\vdash 1~~\mbox{ by Lemma} ~\ref{inverse}(1).
 \end{aligned}$

Also,
 
  $
 \begin{aligned}
 (y\vdash x)\vdash (x^{-1}\perp y^{-1})& \stackrel{L2}= ( (y\vdash x)\vdash x^{-1})\perp y^{-1} \stackrel{ass}= ( (y\vdash (x\vdash x^{-1}))\perp y^{-1}\\&= (y\vdash 1)\perp y^{-1}\\&=1 ~~\mbox{ since} ~~(y^{-1})^{-1}=y\vdash 1~~\mbox{ by Lemma} ~\ref{inverse}(1).
 \end{aligned}$ 
 
Moreover,
 
  $
 \begin{aligned}
 (y\vdash x)\perp(x^{-1}\perp y^{-1})& \stackrel{ass}= ( (y\vdash x)\perp x^{-1})\perp y^{-1} \stackrel{L2}= ( (y\vdash (x\perp x^{-1}))\perp y^{-1}\\&= (y\vdash 1)\perp y^{-1}=1~~\mbox{ by Lemma} ~\ref{inverse}(1). \end{aligned}$ 
 
Finally,
 
  $
 \begin{aligned}
 (x^{-1}\perp y^{-1})\perp(y\vdash x)& \stackrel{ass}= x^{-1}\perp (y^{-1}\perp(y\vdash x))\stackrel{T4}= x^{-1}\perp ((y^{-1}\dashv y)\perp x)\\&=x^{-1}\perp (1\perp x)\\&=1 ~~~\mbox{ since} ~~(x^{-1})^{-1}=x\vdash 1=1\perp x~~\mbox{ by Lemma} ~\ref{inverse}(1).
 \end{aligned}$ 
 
  So $x^{-1}\perp y^{-1}= (y\vdash x)^{-1}=x^{-1}\vdash y^{-1}.$ Therefore $\perp=\vdash$ in $J.$ \\
For 6), it is clear that for $\star\in\{\vdash,\perp,\dashv\},$ and for all $x,y\in A$, we have  
   $$
\phi(x\star y)= ((x\star y)^{-1})^{-1}=(y^{-1}\star x^{-1})^{-1}=(x^{-1})^{-1}\star (y^{-1})^{-1}=  \phi(x)\star  \phi( y),$$ and 
$$
 \begin{aligned}\phi(1)&=\phi(x\vdash x^{-1})= \phi(x)\vdash \phi(x^{-1})=(x^{-1})^{-1}\vdash ((x^{-1})^{-1})^{-1}\\&\stackrel{4)}=(x^{-1})^{-1}\vdash x^{-1}=(x\vdash x^{-1})=(1)^{-1}=1.\end{aligned}$$
  Moreover, it is clear that $\phi(x^{-1})=((x^{-1})^{-1})^{-1}=(\phi(x))^{-1}.$ So $\phi$ is a trigroup homomorphism.
 $\phi$ is onto because $\phi(J)=J,$ since  for any $y:=x^{-1}\in J,$ we have  $\phi(y)= (y^{-1})^{-1}=((x^{-1})^{-1})^{-1} \stackrel{1)}=x^{-1}=y.$ In addition, we have for all $e\in \LieU_A,$ 
$$ \phi(e)\vdash y=\phi(e)\vdash x^{-1}=\phi(e)\vdash \phi(x^{-1})=\phi(e\vdash x^{-1})=\phi(x^{-1})=x^{-1}=y$$ and $$y\dashv  \phi(e)=x^{-1}\dashv  \phi(e)= \phi(x^{-1})\dashv  \phi(e)=\phi(x^{-1}\dashv e)=\phi(x^{-1})=x^{-1}=y.$$ So $\phi(e)\in\LieU_J$ i.e. $\phi(e)=1$ since $J$ is a group. Therefore $\LieU_A\subseteq \ker\phi.$ 
Now let $x\in\ker\phi,$ i.e  $\phi(x)=1. $ Then by $1),$ $1\dashv x=x\vdash 1=(x^{-1})^{-1}=1.$ So for all $y\in A,$ we have 
 $$x\vdash y=x\vdash(1\vdash y)\stackrel{ass}=(x\vdash 1)\vdash y=1\vdash y=y,$$ and  $$y\dashv x=y\dashv(1\vdash x)\stackrel{R1}=y\dashv(1\dashv x)=y\dashv 1=y.$$ It follows that  $x\in\LieU_A.$ Therefore  $\ker\phi\subseteq\LieU_A.$ This completes the proof.
\end{proof}

\begin{Ex}
Consider the trigroup $A:=M\times H$ of Example \ref{ExTri} . Then  $J=\{e\}\times H$ and $\LieU_A=M\times \{1\}$ since   for  all $u\in M$ and $h\in H,$ we have 
$(u,1)\vdash (v,k)=(v,k),~$ $(v,k)\dashv (u,1)=(v,k)$ and  $(e,k^{-1})$ is the inverse of $(v,k)$ for all $u\in M$ and $(v,k)\in A.$ 

\end{Ex}

Let $(A,\vdash,\perp,\dashv)$  be a trigroup, and consider the ternary operation \\$[-,-,-]: A\times A\times A\to A$ defined by $[x,y,z]=(x\perp y)\vdash z\dashv (y^{-1}\perp x^{-1}).$  This operation is a generalization of the conjugation on digroups \cite[Equation (13)]{MK} to trigroups.

\begin{Le}\label{xx1}
Let $(A,\vdash,\perp,\dashv)$  be a trigroup. 
Then the following is true:
\begin{enumerate}

\item[1)]  $e^{-1}\in \LieU_A$ for all $e\in\LieU_A.$
\item[2)] $[x,y,1]=1$ for all  $x,y\in A.$ 

\item[3)] $[e_1,e_2,x]=x$ for all $e_1,e_2\in \LieU_A$ and $x\in A.$ 
 \item[4)] For all $x, y\in A,$ the map $A\stackrel{[x,y,-]}\longrightarrow A$ which associates $[x,y,z]$ to any $z\in A$   is an epimorphism of the underlying trimonoid A.



\end{enumerate}
\end{Le}

\begin{proof}
To prove 1), let $z\in A$ and $e\in\LieU_A.$ Then 

 $
 \begin{aligned} 
   ~e^{-1}\vdash z &=e^{-1}\vdash (1\vdash z) =e^{-1}\vdash ((1\dashv e)\vdash z) \\&\stackrel{ass}=(e^{-1}\vdash (1\dashv e))\vdash z\\&=1\vdash z=z~~\mbox{ since} ~~(e^{-1})^{-1}=e\vdash 1~~\mbox{ by Lemma} ~\ref{inverse}(1).   \end{aligned}$
   
   Similarly, we show that $z\dashv e^{-1}=z.$  Therefore $e^{-1}\in\LieU_A.$

To prove 2), let $x,y\in A,$ and set $\theta=x_1\perp x_2.$ Then 

 $
 \begin{aligned} 
 ~[x,y,1]&=(x\perp y)\vdash 1\dashv (y^{-1}\perp x^{-1})=((x\perp y)\vdash 1)\dashv (x\perp y)^{-1}\\&=(\theta\vdash 1)\dashv\theta^{-1}=1~~\mbox{ since} ~~(\theta^{-1})^{-1}=\theta\vdash 1~~\mbox{ by Lemma} ~\ref{inverse}(1). \end{aligned}$ 
 
 To prove 3), let $z\in A$ and $e_1,e_2\in\LieU_A.$ Then 

 $
 \begin{aligned} 
   ~[e_1,e_2,z]&=(e_1\perp e_2)\vdash z\dashv (e_1^{-1}\perp e_2^{-1})=((e_1\perp e_2)\vdash z)\dashv (e_1^{-1}\perp e_2^{-1})\\&\stackrel{L1}=(e_1\vdash (e_2\vdash z))\dashv (e_1^{-1}\perp e_2^{-1})=(e_2\vdash z)\dashv (e_1^{-1}\perp e_2^{-1}) \\&=z\dashv (e_1^{-1}\perp e_2^{-1})\stackrel{R1}=z\dashv (e_1^{-1}\dashv e_2^{-1})\stackrel{L1}=
z\dashv e_1^{-1}\stackrel{1)}=z.  \end{aligned}$

 To prove 4), let  $x_1,x_2,y,z\in A$ and set $\theta=x_1\perp x_2.$ Then 
\begin{enumerate}
\item[i)]  
 $$
 \begin{aligned} 
 ~[x_1,x_2,y]\dashv [x_1,x_2, z] &=(\theta\vdash y\dashv\theta^{-1})\dashv(\theta\vdash z\dashv\theta^{-1})\\&\stackrel{L2}=\theta\vdash ((y\dashv\theta^{-1})\dashv(\theta\vdash (z\dashv\theta^{-1}))\\&\stackrel{ass}=\theta\vdash (y\dashv(\theta^{-1}\dashv(\theta\vdash (z\dashv\theta^{-1}))))\\&\stackrel{R1}=\theta\vdash (y\dashv(\theta^{-1}\dashv(\theta\dashv (z\dashv\theta^{-1}))))\\&\stackrel{ass}=\theta\vdash (y\dashv((\theta^{-1}\dashv\theta)\dashv (z\dashv\theta^{-1})))\\&=\theta\vdash (y\dashv(1\dashv (z\dashv\theta^{-1})))\\&\stackrel{R1}=\theta\vdash (y\dashv(1\vdash (z\dashv\theta^{-1})))\\&=\theta\vdash (y\dashv (z\dashv\theta^{-1}))\stackrel{ass}=\theta\vdash (y\dashv z)\dashv\theta^{-1}\\&=[x_1,x_2,y\dashv z]. \end{aligned}$$
 \item[ii)]  
  $$
 \begin{aligned} 
 ~[x_1,x_2,y\vdash z] & =\theta\vdash (y\vdash z)\dashv\theta^{-1}\\ &= \theta\vdash ((y\dashv 1)\vdash z)\dashv\theta^{-1}\\&=\theta\vdash ((y\dashv (\theta^{-1} \dashv \theta))\vdash z)\dashv\theta^{-1}\\&\stackrel{ass}=\theta\vdash ((y\dashv \theta^{-1}) \dashv \theta)\vdash z)\dashv\theta^{-1}\\&\stackrel{L1}=\theta\vdash ((y\dashv \theta^{-1}) \vdash (\theta\vdash z))\dashv\theta^{-1}\\&\stackrel{Ass}=((\theta\vdash (y\dashv \theta^{-1})) \vdash (\theta\vdash z))\dashv\theta^{-1}\\&\stackrel{L2}=(\theta\vdash (y\dashv \theta^{-1})) \vdash ((\theta\vdash z)\dashv\theta^{-1})\\&=(\theta\vdash y\dashv \theta^{-1}) \vdash (\theta\vdash z\dashv\theta^{-1})\\&=[x_1,x_2,y]\vdash [x_1,x_2, z] . \end{aligned}.$$ 
\item[iii)] 
  $$ \begin{aligned} 
 ~[x_1,x_2,y\perp z] & =\theta\vdash (y\perp z)\dashv\theta^{-1}\\ &= \theta\vdash (y\perp(1\vdash z))\dashv\theta^{-1}\\&=\theta\vdash (y\perp((\theta^{-1} \perp \theta)\vdash z))\dashv\theta^{-1}\\&\stackrel{L1}=\theta\vdash (y\perp(\theta^{-1} \vdash (\theta\vdash z)))\dashv\theta^{-1} \\ &\stackrel{T4}=\theta\vdash ((y\dashv \theta^{-1}) \perp (\theta\vdash z))\dashv\theta^{-1} \\ &\stackrel{L2}=((\theta\vdash (y\dashv \theta^{-1})) \perp (\theta\vdash z))\dashv\theta^{-1} \\ &\stackrel{R2}=(\theta\vdash (y\dashv \theta^{-1})) \perp ((\theta\vdash z)\dashv\theta^{-1})  \\ &= [x_1,x_2,y]\perp [x_1,x_2, z] . \end{aligned}$$

\end{enumerate}

In addition  we have by 2) that  $[x,y,1]=1$ for all  $x,y\in A.$ That $[x,y,-]$ is onto follows by 3).

\end{proof}

\begin{Le} \label{xyz} Let $x_1,x_2,y_1,y_2,z\in A$ and set $\theta=x_1\perp x_2$

\begin{enumerate}
\item[1)]   $\theta\vdash z=[x_1,x_2,z]\dashv \theta.$
\item[2)]  $ \theta\vdash(y_1\perp y_2)=[x_1,x_2,y_1]\perp(\theta\vdash y_2).$
\item[3)]  $\big[x_1,x_2,[y_1,y_2,z]\big]=[t_1,t_2\dashv \theta,z]$ where $t_1=[x_1, x_2,y_1]$  and $t_2=[x_1, x_2,y_2]$ 

\end{enumerate}
\end{Le}
\begin{proof}

For 1) we have   

  $
 \begin{aligned} ~[x_1,x_2,z]\dashv \theta&= (\theta\vdash z\dashv\theta^{-1})\dashv \theta=((\theta\vdash z)\dashv\theta^{-1})\dashv \theta\\&\stackrel{ass}=(\theta\vdash z)\dashv(\theta^{-1}\dashv \theta)=(\theta\vdash z)\dashv 1=\theta\vdash z.\end{aligned}$
 
For 2), 

  $
 \begin{aligned}  
 \theta\vdash(y_1\perp y_2)&=  \theta\vdash((y_1\dashv 1)\perp y_2)=  \theta\vdash((y_1\dashv (\theta^{-1}\dashv \theta))\perp y_2)\\&\stackrel{ass}= \theta\vdash((y_1\dashv \theta^{-1})\dashv \theta)\perp y_2))\stackrel{T4}=\theta\vdash((y_1\dashv\theta^{-1})\perp (\theta\vdash y_2))\\&\stackrel{L2}=  (\theta\vdash(y_1\dashv\theta^{-1}))\perp (\theta\vdash y_2)=  [x_1,x_2,y_1]\perp (\theta\vdash y_2).\end{aligned}$

For 3),  

$
 \begin{aligned}  
 ~\big[x_1,x_2,[y_1,y_2,z]\big]&=(x_1\perp x_2)\vdash[y_1,y_2,z]\dashv(x_1\perp x_2)^{-1}\\&=\theta\vdash((y_1\perp y_2)\vdash z\dashv(y_1\perp y_2)^{-1})\dashv \theta^{-1} \\&\stackrel{ass}=  ( \theta\vdash(y_1\perp y_2))\vdash z\dashv((y_1\perp y_2)^{-1}\dashv \theta^{-1})   \\&\stackrel{}=  ( \theta\vdash(y_1\perp y_2))\vdash z\dashv ( \theta\vdash(y_1\perp y_2))^{-1}~\mbox{ by Lem.} ~\ref{inverse}(4)\\&\stackrel{2)}=  (t_1\perp(\theta\vdash y_2))\vdash z\dashv (t_1\perp(\theta\vdash y_2))^{-1}\\&\stackrel{1)}= (t_1\perp(t_2\dashv \theta))\vdash z\dashv (t_1\perp(t_2\dashv \theta))^{-1} =[t_1,t_2\dashv \theta,z].\end{aligned}$

  \end{proof}
 
\section{Relating Trigroups to  Leibniz 3-algebras}
Given a field $\mathbb{K}$ of characteristic different to 2, a Leibniz $3$-algebra \cite{CLP} is defined as a $\mathbb{K}$ -vector space $\g$ equipped with a trilinear operation $[-,-,-]: \g^{\otimes 3}\longrightarrow \g$ satisfying the identity 
$$\big[x_1,x_2 [y_1,y_2,y_3]\big]=\big[[x_1,x_2,y_1],y_2,y_3\big]+\big[y_1,[x_1,x_2,y_2], y_3,\big]+\big[y_1,y_2[x_1,x_2,y_3]\big].$$

Recall also from \cite[Definition 2.1]{BGR1} that a $3$-rack  $(R,[-,-,-])$ is a set $R$ endowed with a ternary operation $[-,-,-]: R\times R\times R\longrightarrow R$  such that 

\begin{enumerate}
\item[(3r1)] $\big[x_1,x_2,[y_1,y_2,z]\big]=\big[[x_1,x_2,y_1],[x_1,x_2,y_2],[x_1,x_2,z]\big]$ for all  $x_1,x_2,y_1,y_2,z\in R$
\item[(3r2)] For   $x,y,b\in R,$ there exists  a unique $z\in R$ such that $[x,y,z]=b.$
\end{enumerate}
If in addition there is a distinguish element  $1\in R$ such that\\ $\textit{(3r3)}~[1,1,z]=z~~\mbox{and}~~[x_1,x_2,1]=1~~\mbox{ for all}~~ x_1,x_2,z\in R,$ then $(R,[-,-,-],1)$ is said to be a pointed $3$-rack.

In the next proposition, we equip  a trigroup with a structure of  3-rack. This 
 provides a functor from the category of trigroups to the category of pointed $3$-racks, analogue to the functor from the category of digroups to the category of pointed racks studied in \cite{MK}.
\begin{Pro}
Let $(A,\vdash,\perp,\dashv)$  be a trigroup with distinguish bar-unit 1. Then $(A, [-,-,-])$ is a  3-rack pointed at 1, where the operation $[-,-,-]: A\times A\times A\to A$ is defined by $$[x,y,z]=(x\perp y)\vdash z\dashv (y^{-1}\perp x^{-1}).$$ 
\end{Pro}

\begin{proof}
To verify the axiom (3r1), we have by the property 3) of Lemma \ref{xyz} that \\$[x_1,x_2,[y_1,y_2,z]]=[t_1,t_2\dashv \theta,z]$ where $t_1=[x_1, x_2,y_1],$ $t_2=[x_1, x_2,y_2]$  and \\$\theta=x_1\perp x_2.$
So  

$
 \begin{aligned} 
 ~\big[x_1,x_2,[y_1,y_2,z]\big]&=[t_1,t_2\dashv \theta,z]=(t_1\perp(t_2\dashv \theta))\vdash z\dashv ((t_1\perp(t_2 \dashv\theta))^{-1}\\&\stackrel{R2}=((t_1\perp t_2)\dashv \theta)\vdash z\dashv ((t_1\perp t_2) \dashv\theta)^{-1}\\&=((t_1\perp t_2)\dashv \theta)\vdash z\dashv (\theta^{-1}\dashv (t_1\perp t_2)^{-1})\\&\stackrel{L1}=(t_1\perp t_2)\vdash (\theta\vdash z\dashv (\theta^{-1}\dashv (t_1\perp t_2)^{-1}))~\\&\stackrel{ass}=(t_1\perp t_2)\vdash (\theta\vdash z\dashv \theta^{-1})\dashv (t_1\perp t_2)^{-1}\\&=(t_1\perp t_2)\vdash [x_1,x_2,z]\dashv(t_1\perp t_2)^{-1}=\big[t_1,t_2,[x_1,x_2,z]\big]\\&=\big[[x_1, x_2,y_1],[x_1, x_2,y_2],[x_1,x_2,z]\big].\end{aligned}$
 
To show the axiom (3r2), let   $x,y,b\in A,$ and set $\theta=x\perp y$ and  $z_0=\theta^{-1}\vdash b\dashv\theta.$ Then 

$
 \begin{aligned} 
 ~[x,y,z_0]_{R}&= \theta\vdash (\theta^{-1}\vdash b\dashv\theta)\dashv\theta^{-1}=(\theta\vdash \theta^{-1})\vdash b\dashv(\theta\dashv\theta^{-1})\\&=1\vdash b\dashv 1=b\dashv 1=b.\end{aligned}$
 
   For uniqueness, let $z\in A$ such that   $[x,y,z]_{R}=b$ i.e.  $\theta\vdash z\dashv\theta^{-1}=b.$ 
 So   
 
 $
 \begin{aligned} 
 z\dashv\theta^{-1}&=1\vdash (z\dashv\theta^{-1})=(\theta^{-1}\dashv\theta)\vdash (z\dashv\theta^{-1})\\&\stackrel{L1}=\theta^{-1}\vdash(\theta\vdash (z\dashv\theta^{-1}))=\theta^{-1}\vdash b.\end{aligned}$ 
 
Therefore $$z=z\dashv 1=z\dashv (\theta^{-1}\dashv\theta)=(z\dashv \theta^{-1})\dashv\theta=(\theta^{-1}\vdash b)\dashv\theta=z_0.$$
The  axiom (3r3) is satisfied by the properties 2) and 3) of Lemma \ref{xx1}. \end{proof}




Now, we pay a particular attention to trigroups equipped with a smooth manifold structure.
\begin{De}
A Lie trigroup $(A,\vdash,\perp,\dashv)$  is a smooth manifold $A$ with a trigroup structure such that the operations $\vdash,\perp,\dashv: A\times A\to A$ and the inversion $(.)^{-1}:A\to A$ are smooth mappings. 
\end{De}

Clearly, the pointed 3-rack of Proposition 4.7 induced by a Lie trigroup inherits the smooth manifold structure, and is therefore a Lie 3-rack. It was proven in  \cite[Corollary 3.6]{BGR1} that the tangent space of a Lie 3-rack at the distinguish element 1 has a Leibniz 3-algebra structure. As a  a consequence we have the following  Corollary.

\begin{Co}
Let $(A,\vdash,\perp,\dashv)$  be a Lie trigroup and $(A, [-,-,-])$ its induced Lie 3-rack, and let $\Lieg:=T_1A.$ Then there exists a trilinear mapping  $[-,-,-]_{\Lieg}: \Lieg\times \Lieg\times \Lieg\to \Lieg$ such that $(\Lieg, [-,-,-]_{\Lieg})$ is a Leibniz 3-algebra.
\end{Co}



\begin{center}

\end{center}


\begin{thebibliography}{7}






%

\bibitem{BGR1} G. R. Biyogmam,  Lie $n$-rack, \textit{ C. R. Acad. Sci. Paris}, Ser. I \textbf{349} (2011) 957-960.  




\bibitem{cas1} J. M. Casas, \it{Trialgebras and Leibniz 3-algebras}, {Bol. Soc. Mat. Mex.}, III. {\bf{Ser. 12}} (2006), 165-178.



%



%
%
%

%
%


\bibitem{FR2}  R. Felipe,  \it{Generalized Loday algebras and digroups,} preprint, available at
www.cimat.mx/reportes/enlinea/I-04-01.pdf.
%



%
%
\bibitem {MK} M. K. Kinyon,  \textit{Leibniz algebra, Lie racks and digroups,} J.Lie Theory, \textbf{17} (2007) 99-11.
 \bibitem{KL1} K. Liu, \it{A class of grouplike objects,} www.arXiv.org/math.RA/0311396.
\bibitem{CLP} J. M. Casas,  J. L. Loday,  Pirashvili, T., ``Leibniz $n$-algebras'', \textit{Forum Math.} \textbf{14},  (2002), 189-207.
%
\bibitem{Lo2} {J.-L. Loday}: {\it Une version non commutative des alg\`ebres de Lie: les alg\`ebres de Leibniz}, {L'Enseignement Math\'ematique} {\bf 39} (1993), 269--292.
\bibitem{Lo3} {J.-L. Loday, M.O. Ronco,} {\it Trialgebras and families of polytopes,} In ``Homotopy
Theory: Relations with Algebraic Geometry, Group Cohomology, and
Algebraic K-theory'' Contemporary Mathematics {\bf{346}} (2004), 369-398.








\bibitem{JDP}  Phillips, J. D., \textit{A short Basis for the
Variety of Digroups,} {Semigroup Forum}
OF1-OF5, 2004. 
















\end{thebibliography}
\end{document}